\documentclass[12pt]{amsart}
\usepackage{amscd}
\usepackage{verbatim}
\usepackage{amssymb, amsmath, amsthm, amscd,ifthen}
\usepackage[dvips]{graphics}
\usepackage[cp866]{inputenc}
\usepackage{graphicx}
\usepackage{epsfig}
\usepackage{color}

\usepackage{amsmath,amssymb,amscd,}
\usepackage[all]{xy}
\newtheorem{thm}{Theorem}

\newtheorem{prop}{Proposition}
\newtheorem{lemma}{Lemma}

\date{\today }

\title[ ]{Equilibria of three constrained point charges}
\author{G.Khimshiashvili*, G.Panina$^\dag$, D.Siersma$^\ddag$}

\address{*Ilia State University, Tbilisi, Georgia,
e-mail: giorgi.khimshiashvili@iliauni.edu.ge
\newline
$^\dag$ Institute for Informatics and Automation, St. Petersburg, Russia,
Saint-Petersburg State University, St. Petersburg, Russia,
e-mail:gaiane-panina@rambler.ru.
\newline
$^\ddag$ University of Utrecht,
Utrecht, The Netherlands, e-mail: D.Siersma@uu.nl.}

\begin{document}

\begin{abstract}
We study the critical points of Coulomb energy considered
as a function on configuration spaces associated with certain
geometric constraints. Two settings of such kind are discussed
in some detail. The first setting arises by considering polygons
of fixed perimeter with freely sliding positively charged vertices.
The second one is concerned with triples of positive charges
constrained to three concentric circles. In each of these cases
the Coulomb energy is generically a Morse function. We describe
the minima and other stationary points of Coulomb energy and show that,
for three charges, a pitchfork bifurcation takes place accompanied by
an effect of the Euler's Buckling Beam type.
\end{abstract}

\maketitle \setcounter{section}{0}

\section{Introduction}

We deal with equilibrium configurations of point charges with Coulomb interaction satisfying
certain geometric constraints. Our approach to this topic is similar to the paradigms used in
\cite{khi2}, \cite{khpasi1}. Namely, we consider the Coulomb energy
as a function on a certain configuration space naturally associated with the constraints in question,
and investigate its critical points. The main attention in this paper is given to two specific
problems naturally arising in this setting. The first one deals with identification and
calculation of equilibrium configurations of charges satisfying the given constraints. The second
one, called the \textit{inverse problem}, is concerned with characterizing those
configurations of points for which there exists a collection of non-zero charges such that the given
configuration is a critical point of Coulomb energy on the corresponding configuration space. Such configurations
are called {\it Coulomb equilibria}.

The geometric constraints considered below come from two sources: (i)  Coulomb energy of point charges
freely sliding along a flexible planar contour of fixed length, and (ii) Coulomb energy of concentric
orbitally constrained triples of charges. The first setting was inspired by the concept of "necklace with
interacting beads" introduced and investigated by P. Exner\cite{exn}. This setting has been
considered in a similar situation in our previous paper \cite{khpasi2}.

It turns out that interesting results exist even in the case of three charges. We focus on the minimum
energy and the other stationary points and values. While for almost equal charges the minimum is achieved
on a triangle configuration, it turns out that in both settings the global minimum is achieved in an aligned
situation if one of the charges is much smaller (or bigger, depending on the setting) than the others. The transition between these two states exhibits
the well-known  \textit{supercritical pitchfork bifurcation} accompanied by a \textit{fixing} effect, similar to
the Euler Buckling Beam phenomenon \cite{GS, PS}.

It seems worthy of adding that in most of our considerations the Coulomb forces can be replaced by various other central
forces. The qualitative behaviour will be the same with modifications on the quantitative side.
\medskip

\textbf{Acknowledgment.}{We gratefully acknowledge the hospitality and excellent working conditions at CIRM, Luminy,
where this research was carried out in January of 2016 in the framework of Research in Pairs program.
Research of G.Khimshiashvili was supported by Georgian Science Foundation grant FR/59/5-103/13.
Research of G. Panina was supported by state research grant N 0073-2015-0003.

\section{Configuration spaces and Coulomb energy}

We are basically interested in studying the equilibrium configurations of a system
of repelling charges. As usual equilibria are defined as the critical (stationary)
points of Coulomb energy of an $n$-tuple of points $\{p_i\}$
defined by the formula
$$E=\sum_{i<j}\frac{q_iq_j}{d_{ij}},$$
where $q_i$ are some positive numbers (charges)  assigned to
the points $p_i$, and $d_{ij}$ are the distances $|p_ip_j|$.
To this end the Coulomb energy is considered as a function on a certain configuration
space naturally associated with the given geometric constraints.

Our first configuration space is the space of all labeled $n$-tuples of points\footnote{Indices are modulo $n$, so in the summation we assume $p_{n+1}=p_1$.} with the constraint that the perimeter is not bigger than $1$: $$\overline{M}(n)=\{(p_1,...,p_n)|\  p_i \in \mathbb{R}^2, \  p_1=0, \ \sum_{i=1}^n|p_ip_{i+1}|\leq 1\}/SO(2).$$  Informally, one can think  of a closed rope  with freely sliding positively charged points on it.
Factorization by $SO(2)$ means that we are only interested in the shape of a configuration and ignore orientation preserving motions.
However, we do not identify symmetric $n$-tuples.

\medskip

Let us also introduce the space
$${M}(n)=\{(p_1,...,p_n)|\  p_i \in \mathbb{R}^2, \  p_1=0, \ \sum_{i=1}^n|p_ip_{i+1}|=1\}/SO(2).$$
It is naturally identified with the space of all planar polygons with fixed perimeter (vertices are allowed to coincide)
factorized by orientation preserving motions. The elements of the space $M(n)$ are called either \textit{polygons}
or \textit{configurations.}

For our purposes, it is important to know the topological type of the configuration space.

\begin{thm}\label{ThmProj}
The space ${M}(n)$  is diffeomorphic to  the complex projective space $\mathbb{C}P^{n-2}$.
\end{thm}

Proof.
By definition,  $$\mathbb{C}P^{n-2}=\{(u_1:...:u_{n-1})|\  u_i \in \mathbb{C},
 \hbox{ not all   }  u_i=0  \},$$ assuming that two proportional $(n-1)$-tuples are identified.
 We  add one more term and write $$\mathbb{C}P^{n-2}=\{(u_1:...:u_{n-1}:-\sum_{i=1}^{n-1}u_i )|\  u_i \in \mathbb{C}, \hbox{ not all   }  u_i=0\},$$ with identification
 $$ \ (u_1:...:u_{n-1}:-\sum_{i=1}^{n-1}u_i )=(\lambda \cdot u_1:...: \lambda \cdot u_{n-1}:-\lambda \cdot\sum_{i=1}^{n-1}u_i ) .$$
 This can be interpreted as the space of all $n$-gons with non-zero perimeter.  Indeed,  complex numbers $u_i$ yield  vectors
 in the plane. The factorization by multiplication by complex numbers amounts to factorization of the space of
 polygons by all possible rotations and scalings.\qed

\medskip

The space  $\overline{M}(n)$  is a cone over $M(n)$, and
we have a natural inclusion $\overline{M}(n) \supset M(n).$
All the polygons with non-zero perimeter strictly smaller than $1$ form a manifold  diffeomorphic to $\mathbb{C}P^{n-2}\times \mathbb{R}$, so it makes sense to  speak of critical points of the Coulomb energy $E$.
The informal message of the following proposition is that  the ''sliding charges on a closed rope''  problem  reduces to ''fixed perimeter'' problem.

\begin{prop} The Coulomb energy has no critical points  in $\overline{M}(n)$ outside $M(n)$.
\end{prop}

Proof. Assume $P$ is a critical polygon whose perimeter is strictly smaller than $1$. Its dilation gives a tangent vector
with a non-zero derivative of the Coulomb energy since the dilation strictly increases all pairwise distances between the points.\qed

\medskip

  \bigskip
The second constrained system is associated with a system of three concentric circles.
It is defined by a triple of positive numbers $r_1,r_2,r_3$. The corresponding
configuration space can be represented as

$$T(r_1,r_2,r_3)=\{(p_1,p_2,p_3)|  p_i \in \mathbb{R}^2, \ |p_1|=r_1,\ |p_2|=r_2,\ |p_1|=r_2\}/SO(2).$$

In other words, it is the configuration space of triples of points lying on three concentric circles.
Clearly, $T$ is diffeomorphic to the torus $S^1\times S^1$.
The elements of  $T(r_1,r_2,r_3)$ are also called \textit{configurations}.

\medskip

For both of the above spaces, a configuration  is called \textit{aligned} if all its vertices lie on a line.

  \medskip

  \textbf{Remark.} On each of the above spaces there is a natural involution $\nu$  which takes a configuration to its
  image under reflection  (with respect to a line). For a  configuration $P=\{p_i\}$, we have $\nu(P)=P$ if and only if
  all the points $p_i$ lie on a line. Since  the Coulomb energy is  symmetric with respect the involution, that is, $E(P)=E(\nu(P))$,
  each non-aligned critical configuration comes together with its symmetric image, which is also critical, and has the same Morse index.

\section{Three charges on a contour with given perimeter}

\label{s:perimeter}
In this section we study Coulomb energy on the space $M(3)$ which we also call the space of triangles.
Theorem \ref{ThmProj} implies that $M(3)$ is homeomorphic to the two-sphere $S^2$.

\medskip

Let us first analyze Coulomb energy of three constrained charges in the case when the ambient space
has dimension one. That is, we consider
the configuration space of all aligned triangles with fixed perimeter. Topologically this configuration space is a circle.
Each such  aligned triangle is a segment of length $\frac{1}{2}$  whose endpoints are two of the points $p_i$
and the third point lies between these two (or equals one of them). Two such aligned  triangles are identified
if they differ by reflection.

The Coulomb energy
has three poles on the circle (which correspond to maxima), and  three minima.
Each of the minima corresponds to the choice of the point that lies between two others.
The following observation is well known and easy to prove by direct computation.

\begin{lemma}\label{LemmaDimOne}
Assume that charges $(q_1,q_2,q_3)$ are positioned  on a line at points $p_1,p_2,p_3$ satisfying the fixed perimeter condition.
Then there exists a unique critical point
(which is the  local minimum) of the Coulomb energy with the point  $p_2$ lying between $p_1$ and $p_3$.
The point $p_2$ does not depend on $q_2$ and is given by the proportion

  $$\frac{d_{12}}{d_{23}}=\frac{\sqrt{q_1}}{\sqrt{q_3}}.$$
  The other  local minima (for $p_1$   and   for $p_3$ as intermediate charges) come analogously.
  The global minimum is the one with the smallest intermediate charge.
  All these critical points are non-degenerate Morse points.\qed

\end{lemma}

\medskip
Now let us return to the space $M(3)$. We use a shorter notation
$l_1 = d_{23}, l_2 = d_{31}, l_3 = d_{12}$. In this notation:
$$E = \frac{q_1 q_2}{l_3} + \frac{q_2 q_3}{l_1} +\frac{ q_3 q_1}{ l_2} . $$

\begin{thm}\label{ThmcritCoulomb} For any three positive charges, the Coulomb energy $E$
has  three obvious poles: $p_1=p_2, \ p_2=p_3,$ and $p_1=p_3$. The minima points and the saddle points
depend on the charges. More precisely,
\begin{enumerate}\item If the triple $(\frac{1}{\sqrt{q_1}}:\frac{1}{\sqrt{q_2}}:\frac{1}{\sqrt{q_3}})$
satisfies the strict triangle inequality then we have:
\begin{enumerate}
  \item $E$ has exactly two minima points. They correspond to mutually symmetric triangles whose sidelengths
  come from the proportion
  $(l_1:l_2:l_3)=(\frac{1}{\sqrt{q_1}}:\frac{1}{\sqrt{q_2}}:\frac{1}{\sqrt{q_3}})$.
  \item $E$ has exactly three  saddle Morse points which correspond to aligned configurations from
   Lemma \ref{LemmaDimOne}.
\end{enumerate}
  \item Assume that the triple $(\frac{1}{\sqrt{q_1}}:\frac{1}{\sqrt{q_2}}:\frac{1}{\sqrt{q_3}})$ does not satisfy
  the triangle inequality, namely
 $\frac{1}{\sqrt{q_1}}>\frac{1}{\sqrt{q_2}}+\frac{1}{\sqrt{q_3}}.$   Then we have:
\begin{enumerate}
  \item $E$ has exactly one minimum point. It corresponds  to the aligned configuration described
  in Lemma \ref{LemmaDimOne} with $p_1$ lying between $p_2$ and $p_3$.
  \item $E$ has exactly two  saddle Morse points.  They correspond to  remaining aligned configurations described
  in Lemma \ref{LemmaDimOne}  with either $q_2$ or $q_3$ lying between the two other charges.
\end{enumerate}

\item Assume that the triangle inequality becomes an equality:
\newline $\frac{1}{\sqrt{q_1}}=\frac{1}{\sqrt{q_2}}+\frac{1}{\sqrt{q_3}}.$   Then we have:
\begin{enumerate}
  \item $E$ has exactly one   degenerate minimum point. It corresponds  to the aligned configuration described
  in Lemma \ref{LemmaDimOne} with $p_1$ lying between $p_2$ and $p_3$.  It can be viewed as a meeting point of two minimum points and one saddle point.
  \item $E$ has exactly two non-degenerate saddle Morse points, as in the previous case.
\end{enumerate}
\end{enumerate}
\end{thm}

Proof. First assume that a non-aligned triangle is a critical point of Coulomb energy. Observe that with respect
to coordinates $(l_1, l_2, l_3)$ we have:
 $\nabla (E) = (- \frac{ q_1 q_2}{l_3^2} , -\frac{ q_2,q_3}{l_1^2} , - \frac{q_3 q_1}{l_2^2} )$.  The Langrange multiplier
method with constraint  $l_1 + l_2 + l_3 = 1$ gives:

$$ l_1^2 \; q_1\;  = \; l_2^2 \; q_2 \; = \; l_3^2 \; q_3,$$

which has a unique solution for given $(q_1, q_2, q_3)$. As soon as $(l_1, l_2, l_3)$ satisfy the strict triangle inequality
this is also a solution of our problem (case  1). Otherwise all critical configurations are aligned.

  Next we consider an aligned   configuration. Let us prove that it is critical in the plane if and  only if it is critical in dimension one.
 Assume that  the intermediate vertex is $p_2$, and denote the configuration by $P$. Since we can no longer
use $l_i$'s as local coordinates (as we did in the proof of Lemma \ref{LemmaDimOne}), we introduce other local coordinates on the space $M(3)$ in a neighborhood of $P$. Namely, for a configuration
close to $P$, we may assume that the points $p_1$ and $p_3$ lie on the $x$-axis. The position of $p_2$ uniquely defines the triangle, so we can take
the coordinates of $p_2$ as local coordinates. Since $E(x,y)=E(x,-y)$, we have a critical point as soon as the restriction to the x-axis is critical. Moreover, the Hessian matrix is diagonal  in  $(x,y)$  coordinates.

 Now we know that there  are exactly three  critical points that are aligned configurations.
To complete the proof one has  to use the relation  $\sharp$(maxima) $+$ $\sharp$(minima) $-$ $\sharp$(saddles)$=2.$ That is,

(1) If the triangle inequality holds, there exist two symmetric non-aligned critical points. Elementary calculations show that these are minimum points.
In this case the three aligned configurations are saddles.

(2) Otherwise, the minimum is one of the aligned configurations. Clearly it is the one with the smallest charge as the intermediate point. The other two are saddles.

\medskip

It only remains to check that the critical points are generically non-degenerate. Besides, the below detailed analysis of the determinant of the Hessian gives a better understanding of what's going on.
A similar approach is used in the subsequent section.

Assume that we have an aligned configuration with sidelengths\footnote{We change notation for brevity.} $a$ and $b$ as is depicted in
Figure \ref{FigAlignedTriang}, left. We know from Lemma \ref{LemmaDimOne} that $\frac{\partial^2E}{\partial x^2}>0$, so we are interested in the sign of  $\frac{\partial^2E}{\partial y^2}$. We show that the latter  depends on the charges.

So we take the one-parametric family of  configuration that depends symmetrically on $y$ (see Figure \ref{FigAlignedTriang}, right) and write:
$$l_3=\sqrt{(a-t)^2+y^2}= a-t+\frac{y^2}{2a}+o(y^2),$$  
$$l_1 =\sqrt{(b-s)^2+y^2}= b-s+\frac{y^2}{2b}+o(y^2),$$  
$$l_2=a+b-s-t=1/2-s-t.$$  

The perimeter   stays the same, so
$$2s+2t=\frac{y^2}{2a}+\frac{y^2}{2b}+o(y^2).$$
We may choose any symmetric family, so let us assume that $s=t$, which
implies $t=s=y^2(\frac{1}{8a}+\frac{1}{8b})$.
Substituting this in the Coulomb energy formula we get
$$E= \frac{q_1q_3}{\frac{1}{2}-(\frac{1}{4a}+\frac{1}{4b})y^2} +\frac{q_1q_2}{a-(\frac{1}{8a}+\frac{1}{8b})y^2+\frac{y^2}{2a}}
+\frac{q_3q_2}{b-(\frac{1}{8a}+\frac{1}{8b})y^2+\frac{y^2}{2b}}+ o(y^2). \ \ \ \ (*)$$

It suffices to make qualitative analysis of the formula.  First, we  conclude that generically,  $\frac{\partial^2E}{\partial y^2}$ is non-zero, and therefore, we have a Morse point.  We also see that the second derivative $\frac{\partial^2E}{\partial y^2}$
is positive if $q_2$ is relatively small, since in this case the derivative of  the first term is positive and majorates the other terms.\qed

\begin{figure}[h]\label{FigAlignedTriang}
\centering
\includegraphics[width=10 cm]{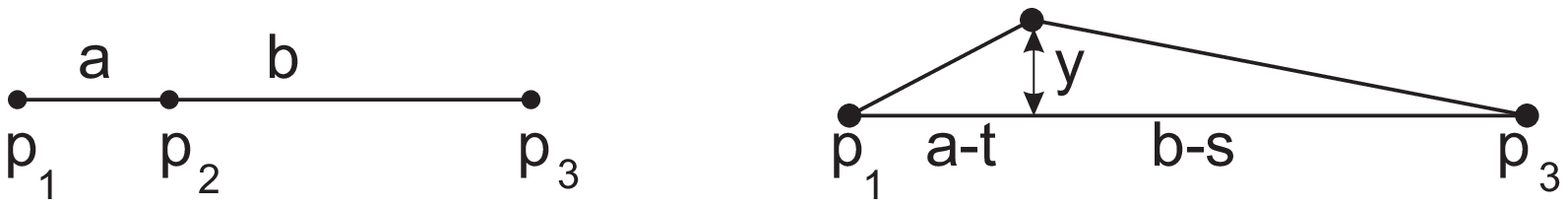}
\caption{}
\end{figure}

\medskip

\begin{figure}\label{Sail}
\centering
\includegraphics[width=12cm]{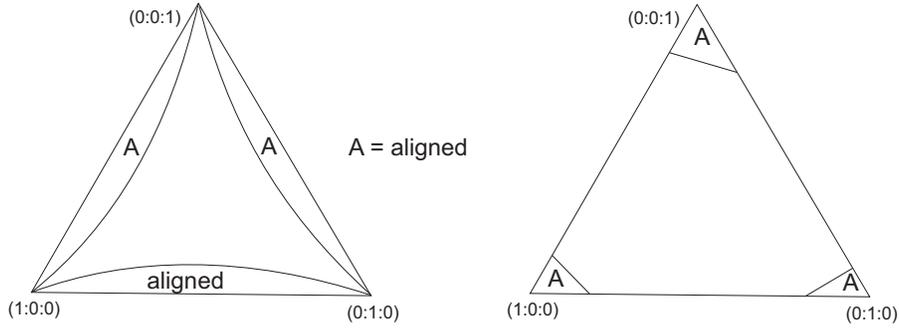}
\caption{The control triangle and the bifurcation curves for the pitchfork bifurcation;
left: Section \ref{s:perimeter}: perimeter constraint, right: Section \ref{s:3circles}: concentric circles}
\end{figure}

\bigskip

Now let us analyze the meaning of the above theorem from the bifurcation theory viewpoint.
By homogeneity we may assume that $q_1+q_2+q_3=1$. Since we only consider positive charges the control space
of our system is a triangle  $\Delta$. In this triangle the  bifurcation set is given by the equalities:
{
$$
\frac{1}{\sqrt{q_1}}=\frac{1}{\sqrt{q_2}}+\frac{1}{\sqrt{q_3}} \;  , \; \frac{1}{\sqrt{q_2}}=\frac{1}{\sqrt{q_3}}+\frac{1}{\sqrt{q_1}} \; , \;  \frac{1}{\sqrt{q_3}}=\frac{1}{\sqrt{q_1}}+\frac{1}{\sqrt{q_2}}
$$}

The middle part of $\Delta$ (see  Fig. 2, left) corresponds to the case (1) of the above theorem:
the minimum is achieved at two mutually symmetric non-degenerate triangles. For the other points of the
triangle  $\Delta$ all critical configurations are aligned.

\begin{figure}[h]\label{Pitch}
\centering
\includegraphics[width=12 cm]{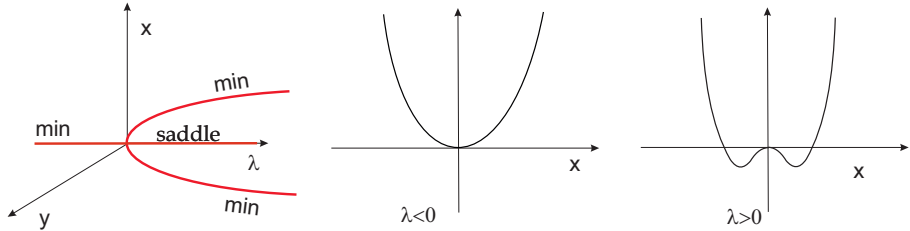}
\caption{The pitchfork bifurcation}
\end{figure}

The situation we encounter here is a typical example of a (supercritical) pitchfork bifurcation. In the single variable case this is
locally described by the potential $f(x,y,\lambda) = \frac{x^4}{4} - \lambda \frac{x^2}{2} + y^2$. Its critical set is given by $x^3 - \lambda x = 0, \; y=0 .$ So it consists of a parabola and a line is the $(x, \lambda)$-plane (see Fig. \ref{Pitch}). The points on the $\lambda$-axis are critical with value $0$ for all $\lambda$: a minumum for $\lambda < 0$ and a saddle point for  $\lambda  > 0$. The points on the parabola are minima with value $ -\frac{1}{2}\lambda ^2$. This is the typical situation for a point on the bifurcation set \cite{GS}.

We can conclude the following. Assume that we have the charge $q_2=0$  whereas the other two charges are non-zero.
Then the configuration $(p_1,p_2,p_3)$ considered as a physical system  becomes aligned with $p_1$ and $p_3$ at the endpoints,
and the position of $p_2$ does not matter. We begin to gradually increase the value of $q_2$. As soon as it is non-zero, $p_2$ takes
some specific position on the aligned configuration  (see Lemma \ref{LemmaDimOne}) (which does not depend on the value of $q_2$
while it is small), stays at the same place (we call this  \textit{fixing effect})
until we cross over the bifurcation set, where it moves away from the line. We get two triangles as minima and the aligned position
persists as a saddle point in the same position.

\medskip

\textbf{Remark}. In view of the above, the inverse problem is now solved straightforwardly.
Namely,

\begin{enumerate}
  \item If three points in the plane are not collinear, their stabilizing charges are defined uniquely up to a scale.
  \item If the three points are aligned, the   stabilizing charges are defined  not uniquely: the charges at the endpoints
  come from Lemma \ref{LemmaDimOne}, whereas the charge of the intermediate point has to be small enough.
\end{enumerate}

\section
 {General fixing effect at aligned positions}

The direct computation  for  $n$-gons (even with $n=4$) leads to complicated equations that do not give a transparent solution.
However some important qualitative observations can be done. An informal message of the below theorem is that for aligned configurations,
there exists  fixing effect of the Euler's Buckling Beam type.

\begin{thm}\label{ThmNtuple}
\begin{enumerate}
  \item No convex critical configuration  $P$ has three charges lying on a line unless $P$ is aligned. In other words, there is no fixing effect for non-aligned convex polygons.
  \item If $q_2,q_3,....,q_{n-1}$ are relatively small\footnote{This means that, for each pair $q_1,q_n$, there exists a positive number $\delta$ such that the statement of the theorem is true for all $q_2,...,q_{n-1}$ smaller than $\delta$.} with respect to the values of $q_1$ and $q_n$, then the Morse index of an  aligned configuration   equals the Morse index of the same configuration  in the case where the ambient space is $\mathbb{R}^1$.
                                                                       In particular, this means that minima points for the ambient space  $\mathbb{R}^1$ remain minima points for the ambient space $\mathbb{R}^2$, and we have the fixing effect for them.

\end{enumerate}
\end{thm}
Proof. (1) Assume the contrary. In this case there exists an infinitesimal motion (that is, an element of the tangent space $T_PM(n)$) which increases some of diagonals with non-zero derivative, whereas  derivatives of all the rest of of pairwise distances is zero, see Fig. \ref{FigConvex}.

\begin{figure}[h]
\centering
\includegraphics[width=4 cm]{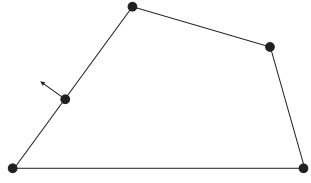}
\caption{This infinitesimal motion  gives non-zero derivative of $E$.
}
\label{FigConvex}
\end{figure}

(2) Let us take an aligned configuration which is a critical point. We can assume that it lies on the $x$-axis. Assume that its endpoints are $p_1$ and $p_n$. In its neighborhood in the space $M(n)$ we choose $(n-2)$ coordinates that correspond to  $x$-coordinates of $p_2,...,p_{n-1}$. The other $(n-2)$ coordinates correspond to $y$-coordinates of the same points. Symmetry arguments $E(\overrightarrow{x},\overrightarrow{y})=E(\overrightarrow{x},-\overrightarrow{y})$ imply that the Hessian matrix is a block  matrix
 $$\left(
     \begin{array}{cc}
       H_1 & 0 \\
       0 & H_2 \\
     \end{array}
   \right),
 $$
 where $H_1$ is the Hessian with respect  to the $x$-coordinates, that is,  the Hessian associated with the one-dimensional ambient space.
 We are going to show  that the Hessian matrix $H_2$  related to the coordinates $y_i$ is positively definite provided that  $q_2,q_3,....,q_{n-1}$ are relatively small. Let us assume that $n=4$; the same approach works for $n>4$.

\begin{figure}[h]
\centering
\includegraphics[width=8 cm]{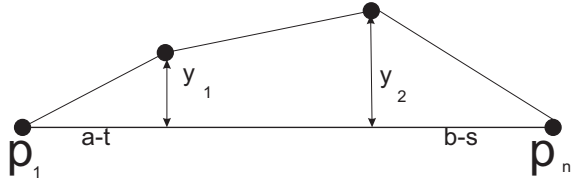}
\caption{Notation for the proof of Theorem \ref{ThmNtuple}
}
\label{FigAlignedN}
\end{figure}

We mimic the proof of  formula $(*)$. Namely, we use the symmetry property of $E$ and choose two shifts $y_1$ and $y_2$.
We can choose the coordinates in any symmetric way with respect to $(y_1,y_2)$, so let us take $s=t$.
Using notations of  Fig. \ref{FigAlignedN}, from the fixed perimeter condition we conclude that
$$s=t=\frac{1}{2}\Big(\frac{y_1^2}{2a}+\frac{y_2^2}{2b}+\frac{(y_1-y_2)^2}{2b}\Big).$$ Therefore
$$E=\frac{q_1q_4}{1-\Big(\frac{y_1^2}{2a}+\frac{y_2^2}{2b}+\frac{(y_1-y_2)^2}{2b}\Big)} +R.$$
Here $R$ is a sum of fractions with smaller numerator and a denominator that quadratically depend on $y_i$.
The first summand  has an obviously positively definite Hessian and majorates the rest of the summands $R$.\qed

\section{Concentric orbitally constrained triples of charges}
\label{s:3circles}
We consider three concentric circles with radii $r_1,r_2,r_3$ and three charges $q_1,q_2,q_3$ .
The correspodning configuration space $T(r_1,r_2,r_3)$ can be identified with a $2$-torus.
Let us fix  some notation  (see  Fig. 6):

$$d_1 = |p_2p_3| \; , \;   d_2 = |p_3p_1| \;  , \; d_3 = |p_1p_2| , $$
$$  \alpha_1 = \angle (p_2p_0p_3) \; , \;  \alpha_2 = \angle (p_3p_0p_1) \; , \;   \alpha_3 = \angle (p_1p_0p_2). \\
$$
Note that $d_i$ depends on $\alpha_i$ via the cosine rule,  e.g.  $d_3^2 = r_1^2+r_2^2-2r_1r_2 \cos \alpha_3$.

\begin{figure}[h]\label{Octopus}
\centering
\includegraphics[width=8 cm]{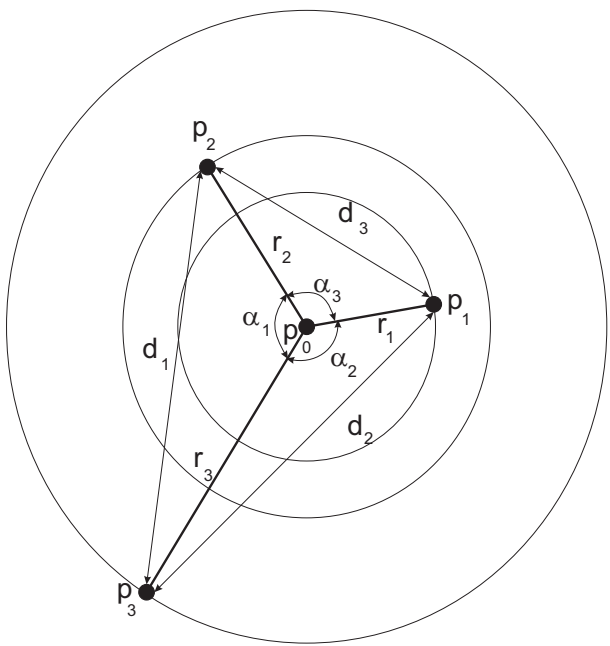}
\caption{A configuration of concentrically constrained triple}
\end{figure}

We suppose that all charges are positive (but similar analysis works also if some of the charges are negative) and
that all $r_i$ are different.
The Coulomb energy
$$  E = \frac{q_1 q_2}{d_3} + \frac{q_2 q_3}{d_1} + \frac{q_3 q_1}{d_2}  $$
is thus a bounded function on the $2$-torus for each $q_1,q_2,q_3$.

\begin{prop} \label{crit3LC}
The critical points of $E$ are given by:

$$ \frac{\sin \alpha_1}{d_1^3r_1q_1} = \frac{\sin \alpha_2}{d_2^3r_2q_2} = \frac{\sin \alpha_3}{d_3^3r_3q_3}$$
\end{prop}

Proof.
This follows by applying Lagrange multiplier  method to $E$ with the constraint $\alpha_1 +\alpha_2 + \alpha_3 = 2 \pi$.\qed

\medskip

An immediate corollary is that the aligned positions (where  $\alpha_i \equiv 0$ modulo $\pi$)  are among the critical points.
There are four aligned  configurations (see Fig. 7); let us  study them first.

\begin{figure}[h]\label{AlignedOct}
\centering
\includegraphics[width=10cm]{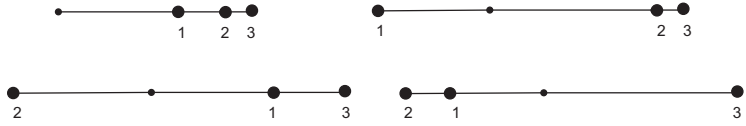}
\caption{Aligned configurations}
\end{figure}
Using a parametrization via $\alpha_1$and $\alpha_2$ we can compute the first and second derivatives of $E$, and evaluate them at
the corresponding critical point. The matrix of second derivatives is (in the aligned case):
$$
\left(
  \begin{array}{cc}
 -q_1q_2 \frac{r_1r_2}{d_3^3} \cos \alpha_3    q_2q_3 \frac{r_2r_3}{d_1^3} \cos \alpha_1 & -q_1q_2 \frac{r_1r_2}{d_3^3} \cos \alpha_3 \\
 q_1q_2 \frac{r_1r_2}{d_3^3} \cos \alpha_3   & -q_1q_2 \frac{r_1r_2}{d_3^3} \cos \alpha_3 -q_3q_3 \frac{r_3r_1}{d_2^3} \cos \alpha_2

  \end{array}
\right)
$$
The Hessian determinant (up to a positive multiple)  becomes:
$$ H(\alpha_1,\alpha_2,\alpha_3)  =  \frac{r_1 }{d_2^3 d_3^3}q_1 \cos \alpha_2 \cos\alpha_3 + \frac{r_2 }{d_3^3 d_1^3}q_2 \cos \alpha_3 \cos\alpha_1 + \frac{r_3  }{d_1^3 d_2^3}q_3 \cos \alpha_1 \cos\alpha_2
$$
Assume that  $\alpha_1 =\alpha_2= \pi \; , \; \alpha_3 = 0$.
In this case we have
$$ d_1 = r_2 + r_3 \; , \; d_2 = r_1 - r_3, \; \; d_3 = r_2 - r_1. $$
Therefore we have a linear form
$$
H(\pi,\pi,0)  = -A_1 q_1 - A_2q_2 + A_3q_3   \; \; \; ( A_i > 0)
$$
It  follows that, for  $q_3$ big and $q_1,q_2$ small, the Hessian is positive and gets negative on the other side of the line $H = 0$ in the affine triangle $\Delta$. The first case corresponds to a local  minimum; the second to a saddle point. Again we meet an example of the {\it pitchfork bifurcation}: fix $q_1,q_2$ small and let $q_3$ start big and decrease next. The system stays at equilibrium minimum in the aligned position until it meets $H=0$; after that the aligned position become a saddle and two minima are born as symmetric triangles.

We give now the final result for all aligned situations.

\begin{thm}\label{TheoremCritConcentric}
All aligned configurations are critical points of the Coulomb energy. More precisely,
\begin{enumerate}
\item If $\alpha_1=\alpha_2=\alpha_3=0$ then we have a non-degenerate maximum (which is the absolute maximum) for all $(q_1,q_2,q_3)$.
\item  The other three cases depend each on $(q_1,q_2,q_3)$: we have a pitchfork bifurcation, which transforms a minimum into a saddle point.
\end{enumerate}
\end{thm}\label{h-lines}
\begin{proof}
$H$ has following form for each of the four cases (all constants below depend on the $r_i$'s and are positive)
\begin{enumerate}
\item $(\alpha_1,\alpha_2,\alpha_3) = (\pi,\pi,0) :  H = -A_1 q_1 - A_2q_2 + A_3q_3$
\item $(\alpha_1,\alpha_2,\alpha_3) = (0,\pi,0) :  H = -B_1 q_1 + B_2q_2 - B_3q_3$
\item $(\alpha_1,\alpha_2,\alpha_3) = (\pi,0,0) : H = C_1 q_1 - C_2q_2 - C_3q_3$
\item $(\alpha_1,\alpha_2,\alpha_3) = (0,0,0) : H = D_1 q_1 + D_2q_2 + D_3q_3$
\end{enumerate}
The sign of the Hessian determines the type of equilibrium. It can be computed from the $q_i$'s and $r_j$'s.
In the first case the linear function $H$ will change sign for positive $q_i$'s. In the last case $H > 0$  for all $q_i$'s.
\end{proof}

Also here the control space (with normalized charges) is a triangle. The zero sets of the linear functions from Theorem \ref{TheoremCritConcentric}
cut out three small triangles around vertices of this triangle (see Fig. 2, right). They don't intersect. Each of the triangles corresponds to a minimum state.

Unlike the constrained problem discussed in Section 3,  the number of
the non-aligned critical configurations is difficult to analyze. They are  solutions of the equation in Proposition \ref{crit3LC}, which seems difficult to solve by hand. We proceed with discussing a few qualitative aspects and mention the results of some tests with Mathematica.

First, the non-degeneracy of $E$ and Betti numbers of the torus imply that there are at least four critical points with sum of their
Morse indices is equal to zero. In case of one minimum, two saddles and one maximum we have an exact Morse function.
From Theorem \ref{TheoremCritConcentric} we have already four critical points, which includes a single maximum. If the three other points
all are  saddle points (where this happens can be computed from the above Hessians) then there must be at least two minima which correspond to
(non-aligned) triangles. They come in couples via symmetry with the same critical value. Computer experiments suggest that there are no more
critical points.

If one of the aligned critical points is a minimum, then two saddles and one maximum  could suffice. This is supported by computer tests. Direct computations show that no more aligned minima appear. The existence of the pitchfork bifurcation also implies the existence of (at least) 2 minima for certain values of the $q_i$'s.

Observe that the Hessian also depends on $r_1,r_2,r_3$. Changing these values will also influence the position of critical configurations,
including the effect of pitchfork bifurcation. An interesting point is also to take the limit to the equal radius case.
In this case we have the following proposition.

\begin{prop}  For three positive equal radii and charges,
the Coulomb energy has a unique minimal value, which is achieved at a equilateral triangle.
This corresponds to two symmetric non-degenerate critical points in the state space.
\qed
\end{prop}

\begin{proof}
We use the equalities from Proposition \ref{crit3LC} and reduce them by direct computation to $\cos \alpha_1 = \cos \alpha_2 = \cos \alpha_3$.
We find as only non-aligned solution the equilateral triangle. The Hessian at the critical points takes value $\frac{25}{144}$, with non-degeneracy a consequence. In the aligned cases at least two points coincide and the energy has  poles with value infinity.

\end{proof}

As a consequence we have that,
  for almost equal radii and almost equal charges, the global minimum is achieved at two mutually symmetric triangles
  that are close to the equilateral one.

\section{Concluding remarks}

The considerations and results of the present paper suggest a few immediate remarks.

(1) The same qualitative behaviour holds for several other central forces, that is, e.g. for the potentials
$$E = \sum_{i<j} \frac{q_i q_j}{d_{ij}^k} \; \; (k > 1), \hbox{ \ or } $$  $$E = \sum_{i<j} q_i q_j \log{d_{ij}} \hbox{ \ \ (logarithmic force)}.$$

(2)  A natural generalization of the setting suggested in this paper arises
if one considers equilibria of point charges lying on several different contours
in the plane. In this case Coulomb energy defines a differentiable function on a torus
so there are topological restrictions on the number of critical points. It would
be interesting to find examples of charges for which the number of
equilibria is minimal or maximal.

(3) There apparently exist many other developments for which the results of this paper
may serve as paradigms.

\medskip

\end{document}